\documentclass[12pt,amsymb,fullpage]{amsart}
\usepackage{amssymb,amscd,pstricks}

\newtheorem{theorem}{Theorem}[section]
\newtheorem{defn}[theorem]{Definition}

\newtheorem{lemma}[theorem]{Lemma}

\newtheorem{eple}[theorem]{Example}
\newtheorem{rmk}[theorem]{Remarks}
\newtheorem{dsc}[theorem]{Discussion}
\newtheorem{nota}[theorem]{Notation}

\newsavebox{\indbin}
\savebox{\indbin}{\begin{picture}(0,0)
\newlength{\gnu}
\settowidth{\gnu}{$\smile$} \setlength{\unitlength}{.5\gnu}
\put(-1,-.65){$\smile$} \put(-.25,.1){$|$}
\end{picture}}

\newcommand{\be}{\begin{enumerate}}
\newcommand{\bd}{\begin{defn}}
\newcommand{\bt}{\begin{theorem}}
\newcommand{\bl}{\begin{lemma}}
\newcommand{\ee}{\end{enumerate}}
\newcommand{\ed}{\end{defn}}
\newcommand{\et}{\end{theorem}}
\newcommand{\el}{\end{lemma}}

\begin{document}
\title{A Simple Proof of a Martingale Representation Theorem Using Nonstandard Analysis}
\author{Tristram de Piro}
\address{Mathematics Department, Harrison Building, Streatham Campus, University of Exeter, North Park Road, Exeter, Devon, EX4 4QF, United Kingdom}
 \email{tdpd201@exeter.ac.uk}
\maketitle
\begin{abstract}
We give a proof of a Martingale Representation Theorem using the methods of nonstandard analysis.
\end{abstract}

We introduce the following spaces;\\

\begin{defn}
\label{prob}

Let $\nu\in{{^{*}\mathcal{N}}\setminus\mathcal{N}}$, and set $\eta=2^{\nu}$. Define;\\

$\overline{\Omega_{\eta}}=\{x\in{^{*}\mathcal{R}}:0\leq x<1\}$\\

$\overline{\mathcal{T}_{\nu}}=\{x\in{^{*}\mathcal{R}}:0\leq x\leq 1\}$\\

We let $\mathcal{C}_{\eta}$ consist of internal unions of the intervals $[{i\over\eta},{i+1\over\eta})$, for $0\leq i\leq \eta-1$, and let $\mathcal{D}_{\nu}$ consist of internal unions $[{i\over\nu},{i+1\over\nu})$, for $0\leq i\leq \nu-1$, together with $\{1\}$\\

We define counting measures $\mu_{\eta}$ and $\lambda_{\nu}$ on $\mathcal{C}_{\eta}$ and $\mathcal{D}_{\nu}$ respectively, by setting $\mu_{\eta}([{i\over\eta},{i+1\over\eta}))={1\over\eta}$, $\lambda_{\nu}(({i\over\nu},{i+1\over\nu}])={1\over\nu}$ and $\lambda_{\nu}(\{1\})=0$\\

We let $(\overline{\Omega_{\eta}},\mathcal{C}_{\eta},\mu_{\eta})$ and $(\overline{\mathcal{T}}_{\nu},\mathcal{D}_{\nu},\lambda_{\nu})$ be the resulting ${*}$-finite measure spaces, in the sense of \cite{L}, and let $(\overline{\Omega_{\eta}},L(\mathcal{C}_{\eta}),L(\mu_{\eta}))$, $(\overline{\mathcal{T}}_{\nu},L(\mathcal{D}_{\nu}),L(\lambda_{\nu}))$ be the associated Loeb spaces.\\

We let $V(\mathcal{C}_{\eta})=\{f:\overline{\Omega_{\eta}}\rightarrow{^{*}{\mathcal{C}}},f(x)=f({[\eta x]\over\eta})\}$  and $W(\mathcal{C}_{\eta})\subset V(\mathcal{C}_{\eta})$ be the set of measurable functions $f:\overline{\Omega_{\eta}}\rightarrow{^{*}{\mathcal{C}}}$, with respect to $\mathcal{C}_{\eta}$, in the sense of \cite{L}. Then
$W(\mathcal{C}_{\eta})$ is a ${*}$-finite vector space over ${^{*}{\mathcal{C}}}$, of dimension $\eta$, (\footnote{\label{vsp} By a ${*}$-vector space, one means an internal set $V$, for which the operations $+:V\times V\rightarrow V$ of addition and scalar multiplication $.:{^{*}\mathcal{C}}\times V\rightarrow V$ are internal. Such spaces have the property that $*$-finite linear combinations ${^{*}}\Sigma_{i\in I}\lambda_{i}.v_{i}$, $(*)$, for a $*$-finite index set $I$, belong to $V$, by transfer of the corresponding standard result for vector spaces. We say that $V$ is a $*$-finite vector space, if there exists a $*$-finite index set $I$ and elements $\{v_{i}:i\in I\}$ such that every $v\in V$ can be written as a combination $(*)$, and the elements $\{v_{i}:i\in I\}$ are independent, in the sense that if $(*)=0$, then each $\lambda_{i}=0$. It is clear, by transfer of the corresponding result for finite dimensional vector space over $\mathcal{C}$, that $V$ has a well defined dimension given by $Card(I)$, see \cite{dep}, even though $V$ may be infinite dimensional, considered as a standard vector space.}). Similarly, we let $V(\mathcal{D}_{\nu})=\{f:\overline{\mathcal{T}_{\nu}}\rightarrow{^{*}{\mathcal{C}}},f(t)=f({[\nu t]\over\nu})\}$  and $W(\mathcal{D}_{\nu})\subset V(\mathcal{D}_{\nu})$ be the set of measurable functions $f:\overline{\mathcal{T}_{\nu}}\rightarrow{^{*}{\mathcal{C}}}$, with respect to $\mathcal{D}_{\nu}$, in the sense of \cite{L}. Then
$W(\mathcal{D}_{\nu})$ is a ${*}$-finite vector space over ${^{*}{\mathcal{C}}}$, of dimension $\nu+1$.

\end{defn}

\begin{defn}
\label{binary}
Given $n\in\mathcal{N}_{>0}$, we let $\Omega_{n}=\{m\in\mathcal{N}:0\leq m<2^{n}\}$, and let $C_{n}$ be the set of sequences of length $n$, consisting of $1$'s and $-1$'s. We let $\theta_{n}:\Omega_{n}\rightarrow{\mathcal{N}}^{n}$ be the map which associates $m\in\Omega_{n}$ with its binary representation, and let $\phi_{n}:\Omega_{n}\rightarrow C_{n}$ be the composition $\phi_{n}=(\gamma\circ\theta_{n})$, where, for $\bar{m}\in {\mathcal{N}}^{n}$, $\gamma(\bar{m})=2.\bar{m}-\bar{1}$. For $\nu\in{{^{*}\mathcal{N}}\setminus\mathcal{N}}$, we let $\phi_{\nu}:\Omega_{\nu}\rightarrow C_{\nu}$ be the map, obtained by transfer of $\phi_{n}$, which associates $i\in{^{*}\mathcal{N}}$, $0\leq i<2^{\nu}$, with an internal sequence of length $\nu$, consisting of $1$'s and $-1$'s. Similarly, for $\eta=2^{\nu}$, we let $\psi_{\eta}:\overline{\Omega_{\eta}}\rightarrow C_{\nu}$ be defined by $\psi_{\eta}(x)=\phi_{\nu}([\eta x])$. For $1\leq j\leq \nu$, we let $\omega_{j}:C_{\nu}\rightarrow \{1,-1\}$ be the internal projection map onto the $j$'th coordinate, and let $\omega_{j}:\overline{\Omega_{\eta}}\rightarrow \{1,-1\}$ also denote the composition $(\omega_{j}\circ\psi_{\eta})$, so that $\omega_{j}\in W(\overline{\Omega_{\eta}})$. By convention, we set $\omega_{0}=1$. For an internal sequence $\overline{t}\in C_{\nu}$, we let $\omega_{\overline{t}}:\overline{\Omega_{\eta}}\rightarrow \{1,-1\}$ be the internal function defined by;\\

$\omega_{\overline{t}}=\prod_{1\leq j\leq\nu}\omega_{j}^{{\overline{t}(j)+1\over 2}}$\\

Again, it is clear that $\omega_{\overline{t}}\in W(\overline{\Omega_{\eta}})$.

\end{defn}

\begin{lemma}
\label{nsindbasis}
The functions $\{\omega_{j}:1\leq j\leq\nu\}$ are $*$-independent in the sense of \cite{A}, (Definition 19), in particular they are orthogonal with respect to the measure $\mu_{\eta}$. Moreover, the functions $\{\omega_{\overline{t}}:\overline{t}\in C_{\nu}\}$ form an orthogonal basis of $V(\overline{\Omega_{\eta}})$, and, if $\overline{t}\neq\overline{-1}$, $E_{\eta}(\omega_{\overline{t}})=0$, and $Var_{\eta}(\omega_{\overline{t}})=1$, where, $E_{\eta}$ and $Var_{\eta}$ are the expectation and variance corresponding to the measure $\mu_{\eta}$.

\end{lemma}

\begin{proof}
According to the definition, we need to verify that for an internal index set $J=\{j_{1},\ldots,j_{s}\}\subseteq\{1,\ldots,\nu\}$, and an internal tuple $(\alpha_{1},\ldots,\alpha_{s})$, where $s=|J|$;\\

$\mu_{\eta}(x:\omega_{j_{1}}(x)<\alpha_{1},\ldots,\omega_{j_{k}}(x)<\alpha_{k},\ldots,\omega_{j_{s}}(x)<\alpha_{s})$\\

$=\prod_{k=1}^{s}\mu_{\eta}(x:\omega_{j_{k}}(x)<\alpha_{k})$ $(*)$\\

Without loss of generality, we can assume that each $\alpha_{j_{k}}>-1$, as if some $\alpha_{j_{k}}\leq -1$, both sides of $(*)$ are equal to zero. Let $J'=\{j'\in J:-1<\alpha_{j'}\leq 1\}$ and $J''=\{j''\in J:1<\alpha_{j''}\}$, so $J=J'\cup J''$. Then;\\

$\mu_{\eta}(x:\omega_{j_{1}}(x)<\alpha_{1},\ldots,\omega_{j_{s}}(x)<\alpha_{s})$\\

\noindent $={1\over \eta}Card(z\in C_{\nu}:z(j')=-1\ for\ j'\in J',z(j'')\in\{-1,1\}\ for\ j''\in J'')$\\

\noindent $={1\over 2^{\nu}}Card(z\in C_{\nu}:z(j')=-1\ for\ j'\in J')={2^{\nu-s'}\over 2^{\nu}}=2^{-s'}$\\

where $s'=Card(J')$. Moreover;\\

$\prod_{k=1}^{s}\mu_{\eta}(x:\omega_{j_{k}}(x)<\alpha_{k})=\prod_{j'\in J'}\mu_{\eta}(x:\omega_{j'}(x)=-1)=2^{-s'}$\\

as $\mu_{\eta}(x:\omega_{j}(x)=-1)={1\over 2}$, for $1\leq j\leq \nu$. Hence, $(*)$ is shown. That $*$-independence implies orthogonality follows easily by transfer, from the corresponding fact, for finite measure spaces, that $E(X_{j_{1}}X_{j_{2}})=E(X_{j_{1}})E(X_{j_{2}})$, for the standard expectation $E$ and independent random variables $\{X_{j_{1}},X_{j_{2}}\}$, $(**)$. Hence, by $(**)$;\\

$E_{\eta}(\omega_{j_{1}}\omega_{j_{2}})=E_{\eta}(\omega_{j_{1}})E_{\eta}(\omega_{j_{2}})=0$, $(j_{1}\neq j_{2})$ $(***)$\\

as clearly $E_{\eta}(\omega_{j})=0$, for $1\leq j\leq \nu$. If $\overline{t}\neq\overline{-1}$, let $J'=\{j':1\leq j'\leq\nu, \overline{t}(j')=1\}$, then;\\

\noindent $E_{\eta}(\omega_{\overline{t}})=E_{\eta}(\prod_{1\leq j\leq\nu}\omega_{j}^{{\overline{t}(j)+1\over 2}})=E_{\eta}(\prod_{j'\in J'}\omega_{j'})=\prod_{j'\in J'}E_{\eta}(\omega_{j'})=0$ $(\sharp)$\\

where, in $(\sharp)$, we have used the facts that $J'\neq\emptyset$ and internal, and a simple generalisation of $(***)$, by transfer from the corresponding fact for finite measure spaces. Hence, $1=\omega_{\overline{-1}}$ is orthogonal to $\omega_{\overline{t}}$, for $\overline{t}\neq\overline{-1}$. If $\overline{t_{1}}\neq\overline{t_{2}}$ are both distinct from $\overline{-1}$, then, if $J_{1}=\{j:1\leq j\leq\nu,\overline{t_{1}}(j)=1\}$ and $J_{2}=\{j:1\leq j\leq\nu,\overline{t_{2}}(j)=1\}$, so $J_{1}\neq J_{2}$ and $J_{1},J_{2}\neq\emptyset$, we have;\\

$E_{\eta}(\omega_{\overline{t_{1}}}\omega_{\overline{t_{2}}})$\\

$=E_{\eta}(\prod_{j\in J_{1}}\omega_{j}.\prod_{j\in J_{2}}\omega_{j})$ $(\sharp\sharp)$\\

$=E_{\eta}(\prod_{j\in (J_{1}\setminus J_{2})}\omega_{j}.\prod_{j\in (J_{2}\setminus J_{1})}\omega_{j})$ $(\sharp\sharp\sharp)$\\

$=E_{\eta}(\prod_{j\in (J_{1}\setminus J_{2})}\omega_{j})E_{\eta}(\prod_{j\in (J_{2}\setminus J_{1})}\omega_{j})=0$ $(\sharp\sharp\sharp\sharp)$\\

In $(\sharp\sharp)$, we have used the definition of $J_{1}$ and $J_{2}$,  and in $(\sharp\sharp\sharp)$, we have used the fact that $(J_{1}\cup J_{2})=(J_{1}\cap J_{2})\sqcup (J_{1}\setminus J_{2})\sqcup (J_{2}\setminus J_{1})$, and $\omega_{j}^{2}=1$, for $1\leq j\leq\nu$. Finally, in $(\sharp\sharp\sharp\sharp)$, we have used the facts that $(J_{1}\setminus J_{2})$ and $(J_{2}\setminus J_{1})$ are disjoint, and at least one of these sets is nonempty, the result of $(\sharp)$ and a similar generalisation of $(***)$. This shows that the functions $\{\omega_{\overline{t}}:\overline{t}\in C_{\nu}\}$ are orthogonal, $(****)$. That they form a basis for $V(\overline{\Omega_{\eta}})$ follows immediately, by transfer, from $(****)$ and the corresponding fact for finite dimensional vector spaces. The final calculation is left to the reader.
\end{proof}

We require the following;\\

\begin{defn}
\label{filtration}
For $0\leq l\leq\nu$, we define $\sim_{l}'$, on $C_{\nu}$, to be the internal equivalence relation given by;\\

$\overline{t}_{1}\sim_{l}'\overline{t}_{2}$ iff $\overline{t}_{1}(j)=\overline{t}_{2}(j)$ $(\forall j\leq l)$\\

We extend this to an internal equivalence relation on $\overline{\Omega}_{\eta}$, which we denote by $\sim_{l}$;\\

$x_{1}\sim_{l}x_{2}$ iff $\psi_{\eta}(x_{1})\sim_{l}\psi_{\eta}(x_{2})$ $(*)$\\

We let $\mathcal{C}_{\eta}^{l}$ be the $*$-finite algebra generated by the partition of $\overline{\Omega}_{\eta}$ into the $2^{l}$ equivalence classes with respect to $\sim_{l}$, $(*)$. As is easily verifed, we have $\mathcal{C}_{\eta}^{l_{l}}\subseteq \mathcal{C}_{\eta}^{l_{2}}$, if $l_{1}\leq l_{2}$, $\mathcal{C}_{\eta}^{0}=\{\emptyset,\overline{\Omega}_{\eta}\}$ and $\mathcal{C}_{\eta}=\mathcal{C}_{\eta}^{\nu}$. For $0\leq l\leq\nu$, we let $W(\mathcal{C}_{\eta}^{l})\subseteq W(\mathcal{C}_{\eta})$ be the set of measurable functions $f:\overline{\Omega}_{\eta}\rightarrow{^{*}\mathcal{C}}$, with respect to $\mathcal{C}_{\eta}^{l}$. We will refer to the collection $\{\mathcal{C}_{\eta}^{l}:0\leq l\leq\nu\}$ of $*$-finite algebras, as the nonstandard filtration associated to $\overline{\Omega}_{\eta}$. We produce a standard filtration $\{\mathfrak{D}_{t}:t\in [0,1]\}$, $(**)$, by following the method of \cite{A}, see Definition 7.14 of \cite{dep}, (replacing the equivalence relation $\sim$ there, by $\sim_{l}$, as given in $(*)$, and being careful to use the index $\nu$ instead of $\eta$. Note that Lemma 7.15 of \cite{dep} still applies in this case.) We also require a slight modification of the construction of Brownian motion in \cite{A}. Namely, we take;\\

$\chi(t,x)={1\over{\sqrt{\nu}}}({^{*}\sum}_{i=1}^{[\nu t]}\omega_{i})$, (\footnote{\label{timezero} We adopt the convention that the sum is zero, when $t=0$})\\

and $W(t,x)={^{\circ}\chi}(t,x)$, $(t,x)\in [0,1]\times\overline{\Omega_{\eta}}$ $(**)$.
\end{defn}

One of the advantages of the non-standard approach to stochastic calculus, is that it allows one to show easily that every stochastic integral is a martingale. We follow the notation from Chapter 7 of \cite{dep}, again using the filtration $(**)$ of Definition \ref{filtration} to replace the one from Definition 7.14, and its subsequent applications;\\

\begin{theorem}
\label{stochastic}
If $g\in{\mathcal G}_{0}$, and $f$ is a $2$-lifting of $g$, then $I(t,x)$, as in Definition 7.20 of \cite{dep}, is equivalent, as a stochastic process, to a martingale, with respect to the filtration $\mathfrak{D}_{t}$, (\footnote{\label{stmart} By which I mean a function $I:[0,1]\times\overline{\Omega_{\eta}}\rightarrow{\mathcal R}$, such that;\\

$(i)$. $I$ is $\mathfrak{B}\times\mathfrak{D}$ measurable (complete product).\\

$(ii)$. $I_{t}$ is measurable with respect to $\mathfrak{D}_{t}$, for $t\in [0,1]$.\\

$(iii)$. $E(|I_{t}|)<\infty$, for $t\in [0,1]$.\\

$(iv)$. $E(I_{t}|\mathfrak{D}_{s})=I_{s}$, if $s<t$ belong to $[0,1]$.\\

$(v)$. For $C\subset\overline{\Omega_{\eta}}$, with $L(\mu_{\eta})(C)=1$, and $x\in C$, the paths $\gamma_{x}:[0,1]\rightarrow{\mathcal R}$,\\
 \indent \ \ \ \ \ \ where $\gamma_{x}(t)=I(t,x)$, are continuous.\\

Most of this definition can be found in \cite{stee}, see also \cite{wil} for a thorough discussion of discrete time martingales. We call a martingale tame if it satisfies the additional conditions that;\\

$(vi)$. $I_{1}\in L^{2}(\overline{\Omega}_{\eta},L(\mu_{\eta}))$ and, for $0\leq s<t\leq 1$;\\

 $\int_{\overline{\Omega}_{\eta}}(I_{t}^{2}-I_{s}^{2})d L(\mu_{\eta})\leq C(t-s)$\\

where $C\in\mathcal{R}_{\geq 0}$\\

$(vii)$ $(UI)$ For $a.a.s$, $0\leq s<1$ and sufficiently small $h>0$, ${[I]_{s+h}-[I]_{s}\over h}$ is strongly uniformly integrable in the sense that that there exists $f:\mathcal{R}\rightarrow\mathcal{R}$, with $f\geq 0$ and $lim_{x\rightarrow\infty}f(x)=0$ such that, for $K>0$, $K\in\mathcal{R}$;

$\int_{{[I]_{s+h}-[I]_{s}\over h}>K}{[I]_{s+h}-[I]_{s}\over h}dL(\mu_{\eta})<f(K)$.\\

where $[I]$ denotes the quadratic variation of the process $I$.

}).

\end{theorem}

\begin{proof}
Let $I'$ be the modification of $I$, as given in the proof of  Theorem 7.25 of \cite{dep}. Then $I'$ and agree $I$ on $[0,1]\times C$, where $P(C)=1$, and $P=L(\mu_{\eta})$, so they are equivalent as stochastic processes. We show that $I'$ is a martingale.\\

 $(i)$ follows from the fact that $I$ is $\mathfrak{B}\times\mathfrak{D}$ measurable, and $I=I'$ a.e $\mu\times L(\mu_{\eta})$, $(*)$. Here, completeness of the product is required.\\

 $(ii)$. By the construction in the proof of Theorem 7.25 of \cite{dep}, $I'_{t}$ is measurable with respect to $\mathfrak{D}_{t}'\subset\mathfrak{D}_{t}$.\\

 $(iii)$. We have, for $t\in [0,1]$;\\

 $\int_{\overline{\Omega_{\eta}}}{I'}^{2}(t,x) d L(\mu_{\eta})=\int_{\overline{\Omega_{\eta}}}I^{2}(t,x) d L(\mu_{\eta})$\\

 $=\int_{\overline{\Omega_{\eta}}}{^{\circ}F}^{2}(t,x) d \mu_{\eta}$\\

 $\leq {^{\circ}\int_{\overline{\Omega_{\eta}}}F^{2}(t,x) d \mu_{\eta}}$\\

 $={^{\circ}\int_{\overline{\Omega_{\eta}}}\int_{0}^{t}f^{2}(t,x) d\lambda_{\nu} d\mu_{\eta}}=||g||^{2}_{L^{2}([0,t]\times\overline{\Omega_{\eta}})}$ $(\dag)$\\

 using $(*)$, Definition 7.20, (see notation in Theorem 7.24), Theorem 3.16 and the proof of Theorem 7.22 in \cite{dep}. Hence $I'_{t}\in L^{2}(\overline{\Omega_{\eta}},\mathcal{C}_{\eta},P)$, so $I'_{t}\in L^{1}(\overline{\Omega_{\eta}},\mathcal{C}_{\eta},P)$, by Holder's inequality, see \cite{Rud}.\\

 $(iv)$.  Suppose $s<t$. We first show that $E(I'_{t}|\mathfrak{D}_{s}')=I'_{s}$, $(\dag\dag)$. Suppose $i\in{^{*}{\mathcal N}}$, with ${i\over\nu}\simeq s$, then we claim that $E(I'_{t}|\sigma(\mathcal{C}_{\eta}^{i})^{comp})=I'_{s}$, $(**)$. As $I_{t}=I'_{t}$ a.e $P$, we have $E(I_{t}'|\sigma(\mathcal{C}_{\eta}^{i})^{comp}))=E(I_{t}|\sigma(\mathcal{C}_{\eta}^{i})^{comp}))$. We can also see that $F_{t}\in SL^{2}(\overline{\Omega_{\eta}},\mathcal{C}_{\eta},\mu_{\eta})$. This follows from the calculation $(\dag)$, Theorem 3.34(i) of \cite{dep}, and the fact that;\\

$\int_{\overline{\Omega_{\eta}}}I^{2}(t,x) d L(\mu_{\eta})=||g||^{2}_{L^{2}([0,t]\times\overline{\Omega_{\eta}})}$\\

by Ito's isometry, as $g\in{\mathcal G}_{0}$. Hence, by Theorem 3.34(iv) of \cite{dep}, $F_{t}\in SL^{1}(\overline{\Omega_{\eta}},\mathcal{C}_{\eta},\mu_{\eta})$, $(***)$. Applying Theorem 7.3(ii) of \cite{dep} and $(***)$;\\

$E(I_{t}|\sigma(\mathcal{C}_{\eta}^{i})^{comp})=E({^{\circ}F_{t}}|\sigma(\mathcal{C}_{\eta}^{i})^{comp})={^{\circ}E(F_{t}|\mathcal{C}_{\eta}^{i})}$\\

We have;\\

$E(F_{t}|\mathcal{C}_{\eta}^{i})=\sum_{j=0}^{i-1}f({j\over\nu},x){\omega_{j+1}\over\sqrt{\nu}}$\\

by ${^{*}}$-independence of the sequence $\{\omega_{j}\}_{0\leq j\leq [\nu t]+1}$. Letting $s'={i-1\over\nu}$, so $s'\simeq s$, $E(F_{t}|\mathcal{C}_{\eta}^{i})=F_{s'}$. We have, using Theorem 7.24 of \cite{dep}, that $I_{s}=I_{s'}$ a.e $P$, so $I'_{s}=I_{s}=I_{s'}$ a.e $P$. As $I'_{s}$ is $\sigma(\mathcal{C}_{\eta}^{i})^{comp}$-measurable, we have $E(I'_{t}|(\mathcal{C}_{\eta}^{i})^{comp})=I'_{s}$, showing $(**)$. As $\mathfrak{D}'_{s}\subset \sigma(\mathcal{C}_{\eta}^{i})^{comp}$, and $I'_{s}$ is $\mathfrak{D}'_{s}$-measurable, we have $E(I'_{t}|\mathfrak{D}'_{s})=I'_{s}$, showing $(\dag\dag)$.\\

If $A\in{\mathfrak D}_{s}$, then, by Lemma 7.15(i) of \cite{dep}, $A\in{\mathfrak D}'_{s_{1}}$, for $s<s_{1}<t$. As $E(I'_{t}|\mathfrak{D}'_{s_{1}})=I'_{s_{1}}$, to show $(iv)$, it is sufficient to prove that;\\

$\int_{A}I'_{s} d L(\mu_{\eta})=lim_{s_{1}\rightarrow s}\int_{A}I'_{s_{1}} d L(\mu_{\eta})$ $(\dag\dag\dag)$\\

To show $(\dag\dag\dag)$, observe that $||I'_{s_{1}}-I'_{s}||_{2}^{2}\leq ||g_{[0,s_{1}]}-g_{[0,s]}||_{2}^{2}$ by $(\dag)$, where $g_{[0,s_{1}]}$ is obtained by truncating the function $g$ to the interval $[0,s_{1}]$, (\footnote{Technically, you need to show that $I_{s_{1}}$ is the non standard stochastic integral of $g_{[0,s_{1}]}$, and then apply Theorem 7.22 of \cite{dep}, however, this is clear by truncating the corresponding lift of $g$.}). Using Holder's inequality and the DCT, we have $lim_{s_{1}\rightarrow s}||I'_{s_{1}}-I'_{s}||_{1}\leq lim_{s_{1}\rightarrow s}||g_{[0,s_{1}]}-g_{[0,s]}||_{1}=0$. Therefore, $(\dag\dag\dag)$ is shown. This proves $(iv)$.\\

$(v)$. This is Theorem 25 of \cite{A}.\\

\end{proof}

We proceed to show the converse, that every martingale can be represented as a stochastic integral, using the nonstandard approach.

\begin{lemma}
\label{filtrationbasis}
For $0\leq l\leq\nu$, a basis of the $*$-finite vector space $W(\mathcal{C}_{\eta}^{l})$ is given by $D_{l}=\bigcup_{0\leq m\leq l}B_{m}$, where, for $1\leq m\leq \nu$, $B_{m}=\{\omega_{\overline{t}}:\overline{t}(m)=1,\overline{t}(m')=-1,m<m'\leq\nu\}$, and $B_{0}=\{\omega_{\overline{-1}}\}$.
\end{lemma}

\begin{proof}
The case when $l=0$ is clear as $\omega_{\overline{-1}}=1$, and using the description of $\mathcal{C}_{\eta}^{0}$ in Definition \ref{filtration}. Using the observation $(*)$ there, we have, for $1\leq l\leq\nu$, that $W(\mathcal{C}_{\eta}^{l})$ is a $*$-finite vector space of dimension $2^{l}$. Using Lemma \ref{nsindbasis}, and the fact that $Card(D_{l})=2^{l}$, it is sufficient to show each $\omega_{\overline{t}}\in D_{l}$ is measurable with respect to $\mathcal{C}_{\eta}^{l}$. We have that, for $1\leq j\leq l$, $\omega_{j}$ is measurable with respect to $\mathcal{C}_{\eta}^{j}\subseteq \mathcal{C}_{\eta}^{l}$. Hence, the result follows easily, by transfer of the result for finite measure spaces, that the product $X_{j_{1}}X_{j_{2}}$, of two measurable random variables $X_{j_{1}}$ and $X_{j_{2}}$ is measurable.

\end{proof}

\begin{defn}
\label{mart}
We define a nonstandard martingale to be a ${\mathcal{D}}_{\nu}\times{\mathcal{C}}_{\eta}$-measurable function $Y:{\overline{\mathcal{T}_{\nu}}}\times{\overline{\Omega}_{\eta}}\rightarrow{^{*}\mathcal{C}}$, such that;\\

$(i)$. For $t\in\overline{\mathcal{T}_{\nu}}$, $Y_{{[\nu t]\over\nu}}$ is measurable with respect to $\mathcal{C}_{\eta}^{[\nu t]}$.\\

$(ii)$. $E_{\eta}(Y_{[\nu t]\over \nu}|\mathcal{C}_{\eta}^{[\nu s]})=Y_{{[\nu s]\over \nu}}$, for $(0\leq s\leq t\leq 1)$.\\

$(iii)$. $E_{\eta}(|Y_{[\nu t]\over \nu}|)$ is finite.\\

We say that $Y$ is $S$-continuous, if there exists $C\subset {\overline{\Omega}_{\eta}}$ with $L(\mu_{\eta})(C)=1$, such that for $x\in C$, $Y(t,x)\simeq Y(s,x)$, when $s\simeq t$, and each $Y(t,x)$ is near standard. We say that $Y$ has infinitesimal increments if, for all $x\in{\overline{\Omega}_{\eta}}$, and $t\in{\overline{\mathcal{T}_{\nu}}}$, $t\neq 1$, $Y({[t\nu]+1\over \nu},x)\simeq Y({[t\nu]\over \nu},x)$.\\

\end{defn}

\begin{lemma}
\label{rep}
Let $Y:\overline{\mathcal{T}_{\nu}}\times\overline{{\Omega}_{\eta}}\rightarrow{^{*}\mathcal{R}}$ be a $\mathcal{D}_{\nu}\times\mathcal{C}_{\eta}$-measurable function, satisfying $(i)$ and $(ii)$ of Definition \ref{mart}, then;\\

$Y_{t}(x)=\sum_{j=0}^{[\nu t]}c_{j}(t,x)\omega_{j}(x)$ $(*)$\\

where $c_{0}:[0,1]\times\overline{{\Omega}_{\eta}}\rightarrow{^{*}\mathcal{C}}$ is $\mathcal{D}_{\nu}\times\mathcal{C}_{\eta}^{0}$-measurable, $c_{j}:[{j\over\nu},1]\times\overline{{\Omega}_{\eta}}\rightarrow{^{*}\mathcal{C}}$ is $\mathcal{D}_{\nu}\times\mathcal{C}_{\eta}^{j-1}$-measurable, for $1\leq j\leq \nu$, and $c_{0}(s,x)=c_{0}(t,x)$, for $0\leq s\leq t\leq 1$, $c_{j}(s,x)=c_{j}(t,x)$, for ${j\over\nu}\leq s\leq t\leq 1$. Conversely, if $\{c_{j}:0\leq j\leq \nu\}$ is a collection of functions satisfying the above conditions, then the definition $(*)$ produces a $\mathcal{D}_{\nu}\times\mathcal{C}_{\eta}$-measurable function, satisfying $(i)$ and $(ii)$ of Definition \ref{mart}.

\end{lemma}

\begin{proof}
Using $(ii)$, we have that $E_{\eta}(Y_{t})=E_{\eta}(Y_{t}|\mathcal{C}_{\eta}^{0})=Y_{0}$. Replacing $Y_{t}$ by $Y_{t}-Y_{0}$, we can, without loss of generality, assume that $E_{\eta}(Y_{t})=0$, for $t\in {{^{*}}[0,1]}$. By $(i)$ and Lemma \ref{filtrationbasis};\\

$Y_{t}=\sum_{j=1}^{[\nu t]}c_{j}(t,x)\omega_{j}(x)$\\

where;\\

$c_{j}(t,x)=\sum_{a=0}^{j-1}\sum_{i_{0}<\ldots<i_{a};0}^{j-1}p_{j}^{(i_{0},\ldots,i_{a})}(t)\omega_{i_{0}}\ldots\omega_{i_{a}}(x)$\\

Clearly, $c_{j}$ is $\mathcal{D}_{\nu}\times \mathcal{C}_{\eta}^{j-1}$-measurable.
Again, using $(ii)$, and the fact that $c_{k}(t,x)\omega_{k}$ is orthogonal to the basis $D_{[\nu s]}$ of $W(\mathcal{C}_{\eta}^{[\nu s]}$, for $[\nu s]<k\leq [\nu t]$, $(\dag)$, we have;\\

$\sum_{j=1}^{[\nu s]}c_{j}(t,x)\omega_{j}(x)=\sum_{j=1}^{[\nu s]}c_{j}(s,x)\omega_{j}(x)$\\

Equating coefficients, and using the fact that $D_{j}$ is a basis for $W(\mathcal{C}_{\eta}^{j})$, for $1\leq j\leq [\nu s]$, we obtain $c_{j}(s,x)=c_{j}(t,x)$, for all ${j\over\nu}\leq s\leq t\leq 1$.\\

 The converse is easy to check. $(i)$ is obtained, observing that for $t\in\mathcal{T}_{\nu}$, all the functions $c_{j,t}$ and $\omega_{j}$ are measurable with respect to $\mathcal{C}_{\eta}^{[\nu t]}$, for $0\leq j\leq [\nu t]$. To obtain $(ii)$, just take the conditional expectation of $(*)$ and make the observation $\dag$ again.\\

\end{proof}

\begin{lemma}
\label{lift}
Let $X$ be a martingale, see footnote \ref{stmart} for the definition, with the extra condition that $X_{1}\in L^{2}(\overline{\Omega}_{\eta})$, then there exists a nonstandard martingale $\overline{X}$, see Definition \ref{mart},  with ${^{\circ}}({\overline{X}}_{t})=X_{^{\circ}t}$, for $t\in{\overline{\mathcal{T}_{\nu}}}$, a.e $L(\mu_{\eta})$, and such that the sequence $\{\overline{X}_{i\over \nu}:0\leq i\leq \nu\}\subset SL^{2}(\overline{\Omega_{\eta}},\mu_{\eta})$. Moreover, $\overline{X}$ is $S$-continuous, and we can take $\overline{X}$ to have infinitesimal increments.
\end{lemma}

\begin{proof}
By $(i)$ of footnote \ref{stmart}, we have $X$ is $\mathfrak{B}\times\mathfrak{D}$-measurable. We claim that $X\in L^{1}([0,1]\times\overline{\Omega_{\eta}})$, $(*)$. Without loss of generality, we can assume that $X\geq 0$, (\footnote{\label{positive} In order to see this, it is sufficient to show that $X^{+}$ is a martingale, $(*)$. We have $X=X^{+}-X^{-}$, and, by $(iv)$, for $0\leq t\leq 1$;\\

$X_{t}=X_{t}^{+}-X_{t}^{-}=E(X_{1}|\mathfrak{D}_{t})=E(X_{1}^{+}-X_{1}^{-}|\mathfrak{D}_{t})=Y_{t}-Y'_{t}$ $(**)$\\

where $Y_{t}=E(X_{1}^{+}|\mathfrak{D}_{t})$ and $Y'_{t}=E(X_{1}^{-}|\mathfrak{D}_{t})$. It follows easily, modifying $Y$ to $Y^{1}$, and $Y'$ to $Y^{',1}$, a.e $L(\lambda_{\nu})\times L(\mu_{\eta})$, if necessary, and, using the tower law and definition of conditional expectations, see \cite{wil}, that $Y,Y'$ are martingales and $Y,Y'\geq 0$. We then have, by $(**)$, that $X_{t}^{+}=Y_{t}$ and $X_{t}^{-}=Y'_{t}$ a.e $L(\mu_{\eta})$. Hence, $(*)$ is shown.})
Then $(*)$ follows from the fact that, for $0\leq t\leq 1$, $E(X_{t})=E(X_{t}|\mathfrak{D}_{0})=X_{0}$, by $(iv)$ of footnote \ref{stmart}, and so;\\

$\int_{[0,1]\times\overline{\Omega_{\eta}}}X(t,x)d(L(\lambda_{\nu})\times L(\mu_{\eta}))=X_{0}<\infty$\\

by $(iii)$ of footnote \ref{stmart} and Fubini's theorem, see \cite{Rud}. By the hypothesis that $X_{1}\in L^{2}(\overline{\Omega}_{\eta})$, and using Theorem 7 of \cite{A}, see also Theorems 3.31 and 3.34 of \cite{dep}, we can find $\overline{V}\in SL^{2}(\overline{\Omega_{\eta}},\mu_{\eta})$, with $({^{\circ}\overline{V}})=X_{1}$, a.e $L(\mu_{\eta})$, $(\dag)$. We now define $\overline{X}:\overline{T}_{\nu}\times\overline{\Omega}_{\eta}\rightarrow{^{*}}\mathcal{C}$ by taking $\overline{X}(t,x)=(E_{\eta}(\overline{V}|\mathcal{C}_{\eta}^{[\nu t]}))(x)$. We may assume that $\overline{X}$ is $\mathcal{D}_{\nu}\times\mathcal{C}_{\eta}$ measurable, by the definition of $E_{\eta}( | )$, see footnote 25 of Chapter 7, \cite{dep}, and transfer of the corresponding result for finite measure spaces. Then, by Theorem 7.3 of \cite{dep};\\

$(^{\circ}\overline{X})(t,x)={^{\circ}}(E_{\eta}(\overline{V}|\mathcal{C}_{\eta}^{[\nu t]}))(x)=E((^{\circ}\overline{V})|\sigma(\mathcal{C}_{\eta}^{[\nu t]})^{comp})$ $(**)$\\

Moreover, if $A\in\sigma(\mathcal{C}_{\eta}^{[\nu t]})^{comp}$, we have;\\

$\int_{A}X_{^{\circ}t}d L(\mu_{\eta})=lim_{t'\rightarrow {^{\circ}t}}\int_{A}X_{t'}d L(\mu_{\eta})=\int_{A}X_{1}d L(\mu_{\eta})$ $(***)$\\

using $(iv)$,$(v)$ of footnote \ref{stmart} and the result of $(*)$ to apply the $DCT$. Hence, as $\mathfrak{D}_{^{\circ}t}\subset \sigma(\mathcal{C}_{\eta}^{[\nu t]})^{comp}\subset\mathfrak{D}_{t'}$, for $0\leq {^\circ{t}}<t'$, using $(**)$ in Definition \ref{filtration}, we have;\\

$E((^{\circ}\overline{V})|\sigma(\mathcal{C}_{\eta}^{[\nu t]})^{comp})=E((^{\circ}\overline{V})|\mathfrak{D}_{^{\circ}t})=E(X_{1}|\mathfrak{D}_{^{\circ}t})=X_{^{\circ}t}$\\

by $(***)$, $(\dag)$ and $(iv)$ of footnote \ref{stmart}. By $(**)$, we then have $({^{\circ}}\overline{X}_{t})=X_{^{\circ}t}$, a.e $L(\mu_{\eta})$. We now verify conditions $(i),(ii),(iii)$ of Definition \ref{mart}. $(i)$ is clear by Definition of $\overline{X}$ and footnote 25 of Chapter 7, \cite{dep}. $(ii)$ follows by transfer of the tower law for the conditional expectation $E_{\eta}(  |  )$, see again footnote 25 of Chapter 7. $(iii)$ follows immediately from the fact that $\overline{V}\in SL^{2}(\overline{\Omega_{\eta}},\mu_{\eta})$, and;\\

$|E_{\eta}(\overline{X}_{t})|= |E_{\eta}(\overline{V})|\leq E_{\eta}(|\overline{V}|)\leq ||\overline{V}||_{SL^{2}}\simeq ||X_{1}||_{L^{2}}<\infty$ (for $t\in\mathcal{T}_{\nu}$)\\

by transfer of Holders inequality, the definition of $E_{\eta}(  |  )$, and property $(ii)$ in Definition \ref{mart}. Finally, using Theorem 7.3 of \cite{dep}, we have that the sequence $\{\overline{X}_{i\over \nu}:0\leq i\leq \nu\}\subset SL^{2}(\overline{\Omega_{\eta}},\mu_{\eta})$. The $S$-continuity claim follows from the proof of Theorem 8.1 in \cite{HP2}. We omit the details. For the final claim, we modify $\overline{X}$ to obtain the final condition, while preserving the other properties. For $n\in{{^{*}}\mathcal{N}}$, we let;\\

$V_{n}=\{x:\exists t(|\Delta\overline{X}(t,x)|\geq {1\over n})\}$, (\footnote{We use the notation $\Delta\overline{X}(t,x)$ to denote the increment $\overline{X}(t+{1\over\nu},x)-\overline{X}(t,x)$, for $0\leq t\leq 1-{1\over\nu}$})\\

By $S$-continuity of $\overline{X}$, we have that the internal set $A=\{n\in{{^{*}}\mathcal{N}}:\mu_{\eta}(V_{n})\leq {1\over n}\}$ contains $\mathcal{N}$, hence, it contains an infinite element $\kappa$. For $x\in V_{\kappa}$, we let $\tau(x)$ be the first $t$ such that $|\Delta\overline{X}(t,x)|\geq {1\over\kappa}$ and let $\tau(x)=1$ otherwise. We let $\overline{W}$ be the internal process defined by;\\

$\overline{W}_{0}=\overline{X}_{0}$\\

$\Delta\overline{W}(x,t)=\Delta\overline{X}(x,t)$, if $t<\tau(x)$.\\

$\Delta\overline{W}(x,t)=0$, if $t\geq \tau(x)$.\\

We claim that $\overline{W}$ is a nonstandard martingale in the sense of Definition \ref{mart}. For $(i)$, by hyperfinite induction, and the fact that $\overline{W}_{0}=\overline{X}_{0}$, it is sufficient to show that if $\overline{W}_{i-1\over\nu}$ is measurable with respect to $\mathcal{C}_{\eta}^{i-1}$, then $\overline{W}_{i\over\nu}$ is measurable with respect to $\mathcal{C}_{\eta}^{i}$, for $1\leq i\leq \nu$, $(\dag\dag)$. If $x\sim_{i} x'$, we have that ${i-1\over \nu}<\tau(x)$ iff ${i-1\over\nu}<\tau(x')$, as this is an internal definition depending only on information up to time ${i\over\nu}$, hence must contain the equivalence class $[x]_{\sim_{i}}$. In this case, we have that $\overline{W}(x,{i\over\nu})=\overline{W}(x,{i-1\over\nu})+\Delta\overline{X}(x,{i-1\over\nu})$, which is constant on $[x]_{\sim_{i}}$, using the inductive hypothesis and measurability of $\overline{X}$. The case when ${i-1\over \nu}\geq \tau(x)$ is similar. Hence, $(\dag\dag)$ and $(i)$ are shown. For $(ii)$, it is sufficient to show that if $x\in\overline{\Omega}_{\eta}$, then;\\

 $\int_{[x]_{\sim_{i-1}}}\overline{W}_{i-1\over \nu}d\mu_{\eta}=\int_{[x]_{\sim_{i-1}}}\overline{W}_{i\over \nu}d\mu_{\eta}$, for $1\leq i\leq \nu$ $(\dag\dag\dag)$\\

 Clearly, if ${i-1\over\nu}\geq \tau(x')$, for all $x'\in [x]_{\sim_{i-1}}$, then $\Delta\overline{W}(x,{i-1\over\nu})|_{[x]_{\sim_{i-1}}}=0$, and the result $(\dag\dag\dag)$ follows trivially. Similarly, if ${i-1\over\nu}<\tau(x')$, for all $x'\in [x]_{\sim_{i-1}}$, then $\overline{W}|_{[x]_{\sim_{i-1}}\times [{i-1\over \nu},{i+1\over \nu})}=\overline{X}|_{[x]_{\sim_{i-1}}\times [{i-1\over \nu},{i+1\over \nu})}$, and the result $(\dag\dag\dag)$ follows from the martingale property of $\overline{X}$. We can, therefore, write $[x]_{\sim_{i-1}}=[x_{1}]_{\sim_{i}}\cup [x_{2}]_{\sim_{i}}$, and assume that ${i-1\over \nu}<\tau(x')$, for all $x'\in [x_{1}]_{\sim_{i}}$, and ${i-1\over \nu}\geq \tau(x')$, for all $x'\in [x_{2}]_{\sim_{i}}$. If ${i-2\over \nu}\geq \tau(x')$, for all $x'\in [x_{2}]_{\sim_{i}}$, then the same must hold for all $x'\in [x_{1}]_{\sim_{i}}$, contradicting the assumption. Hence, we can also assume that ${i-2\over \nu}<\tau(x')$, for all $x'\in [x_{2}]_{\sim_{i}}$. It follows that $|\Delta\overline{X}(x,{i-1\over\nu})|_{[x_{1}]_{\sim_{i}}}|\leq {1\over\kappa}$ and $|\Delta\overline{X}(x,{i-1\over\nu})|_{[x_{2}]_{\sim_{i}}}|> {1\over\kappa}$, but this contradicts the martingale property $(\dag\dag\dag)$ for $\overline{X}$. Hence, this case can't happen, so $(\dag\dag\dag)$ and $(ii)$ is shown. Property $(iii)$ follows from the fact that $\overline{W}_{1}\in SL^{2}(\overline{\Omega}_{\eta})$, $(\dag\dag\dag\dag)$, which we show below, and the inequality;\\

 $E_{\eta}(|\overline{W}_{\nu t\over \nu}|)\leq E_{\eta}(\overline{W}_{\nu t\over \nu}^{2})^{1\over 2}\leq E_{\eta}(\overline{W}_{1}^{2})^{1\over 2}$.\\

which uses Cauchy-Schwartz, and the martingale property $(ii)$. By construction $\overline{W}$ has infinitesimal increments. As we are only modifying $\overline{X}$ inside $V_{\kappa}\times {\mathcal T}_{\nu}$, where $L(\mu_{\eta})(V_{\kappa})=0$ it is clear that $S$-continuity is preserved. Similarly, we must have that ${^{\circ}}({\overline{W}}_{t})=X_{^{\circ}t}$, for $t\in{\overline{\mathcal{T}_{\nu}}}$, a.e $L(\mu_{\eta})$. It remains to show $(\dag\dag\dag\dag)$. By the above remark on modification, it is sufficient to show that $\int_{V_{\kappa}}\overline{W}_{1}^{2}d\mu_{\eta}\simeq 0$. We can define a relation on $\overline{\Omega}_{\eta}$ by $x\sim x'$ if $x'\in [x]_{\tau(x)-1}$. If $x\sim x'$, then, by the above discussion, $\tau(x)=\tau(x')$, and so $\sim$ defines an equivalence relation. We clearly have that $V_{\kappa}=\bigcup_{1\leq j\leq r}[x_{j}]_{\sim}$ is an internal union of such equivalence classes. A simple calculation gives that;\\

$\int_{V_{\kappa}}\overline{W}_{1}^{2}d\mu_{\eta}={{^{*}}\sum}_{1\leq j\leq r}\int_{[x_{j}]_{\sim}}\overline{W}_{1}^{2}d\mu_{\eta}$\\

$={{^{*}}\sum}_{1\leq j\leq r}\int_{[x_{j}]_{\tau(x_{j})-1}}\overline{X}_{\tau(x_{j})-1}^{2}d\mu_{\eta}$\\

$\leq {{^{*}}\sum}_{1\leq j\leq r}\int_{[x_{j}]_{\tau(x_{j})-1}}\overline{X}_{1}^{2}d\mu_{\eta}$\\

$=\int_{V_{\kappa}}\overline{X}_{1}^{2}d\mu_{\eta}\simeq 0$\\

where we have used the definition of $\overline{W}$, and the calculation of Theorem 12(ii) in \cite{A}. This gives the result.\\

\end{proof}

\begin{lemma}
\label{tame}
Let $X$ be a tame martingale, and let $\overline{X}$ be as in Lemma \ref{lift}. Then we can find $\kappa\in {{{^{*}}\mathcal{N}}\setminus\mathcal{N}}$ such that $\kappa|\nu$, and for all $t\in\mathcal{T}_{\nu}$;\\

$\int_{\overline{\Omega}_{\eta}}(\overline{X}_{t+{1\over\kappa}}^{2}-\overline{X}_{t}^{2})d\mu_{\eta}\leq {C+1\over\kappa}$\\

where $C\in\mathcal{R}_{\geq 0}$ is as given in footnote \ref{stmart}. Moreover, we can find $D\subset\overline{\Omega}_{\eta}$, with $\mu_{\eta}(D)\simeq 1$, $E\subset\mathcal{T}_{\nu}$ with $\mu_{\eta}(E)\simeq 0$, such that for all $t\in {\mathcal{T}_{\nu}\setminus E}$;\\

$1_{D}\kappa([\overline{X}]_{t+{1\over\kappa}}-[\overline{X}]_{t})\in SL^{1}(\overline{\Omega}_{\eta},\mu_{\eta})$\\

\end{lemma}

\begin{proof}
Without loss of generality we can assume that $n|\nu$, for all $n\in\mathcal{N}$. If $t\in\mathcal{T}_{\nu}$ and $n\in\mathcal{N}$, we have that $\{\overline{X}_{t},\overline{X}_{t+{1\over n}}\}\subset SL^{2}(\overline{\Omega}_{\eta},\mu_{\eta})$, hence $(\overline{X}_{t+{1\over n}}^{2}-\overline{X}_{t}^{2})\in SL^{1}(\overline{\Omega}_{\eta},\mu_{\eta})$. We, therefore, have that;\\

${^{\circ}}(\int_{\overline{\Omega}_{\eta}}(\overline{X}_{t+{1\over n}}^{2}-\overline{X}_{t}^{2})d\mu_{\eta})$\\

$=\int_{\overline{\Omega}_{\eta}}({^{\circ}}(\overline{X}_{t+{1\over n}})^{2}-{^{\circ}}(\overline{X}_{t})^{2})d L(\mu_{\eta})$\\

$=\int_{\overline{\Omega}_{\eta}}(X_{{^{\circ}}t+{1\over n}}^{2}-X_{{^{\circ}}t}^{2})d L(\mu_{\eta})\leq {C\over n}$\\

It follows that;\\

$\int_{\overline{\Omega}_{\eta}}(\overline{X}_{t+{1\over n}}^{2}-\overline{X}_{t}^{2})d\mu_{\eta}\leq {C+1\over n}$\\

As this holds for all $n\in\mathcal{N}$, and the property is internal, we can find an infinite $\kappa|\nu$, such that;\\

$\int_{\overline{\Omega}_{\eta}}(\overline{X}_{t+{1\over \kappa}}^{2}-\overline{X}_{t}^{2})d\mu_{\eta}\leq {C+1\over \kappa}$\\

for all $t\in\mathcal{T}_{\nu}$, as required, for the first part.\\

For the second condition, using Proposition 4.4.12 in \cite{AFKL}, we can assume that there exists $C\subset\overline{\Omega}_{\eta}$, with $L(\mu_{\eta})(C)= 1$, such that $[\overline{X}]$ lifts the standard process $[X]$ on $C\times\mathcal{T}_{\nu}$. For ease of notation, for $m\in\mathcal{N}$ and $t\in\mathcal{T}_{\nu}$, let $[\overline{X}]_{t,m}$ denote the increment $m([\overline{X}]_{t+{1\over m}}-[\overline{X}]_{t})$ and $[X]_{t,m}$ the corresponding standard increment, for $t\in [0,1]$. We clearly have that ${^{\circ}[\overline{X}]_{t,m}}=[X]_{{^{\circ}}t,m}$ on $C\times\mathcal{T}_{\nu}$. Choose a sequence of $\{C_{m}:m\in\mathcal{N}\}$, with $C_{m}\subset C$, such that each $C_{m}\in\mathcal{C}_{\eta}$ and $\mu_{\eta}(C_{m})=1-{1\over m}$. As $C_{m}$ is internal and $[\overline{X}]$ lifts $X$ on $C_{m}\times\mathcal{T}_{\nu}$, by compactness, we must have that $[\overline{X}]$ is bounded on $C_{m}\times\mathcal{T}_{\nu}$, $|[\overline{X}]|\leq D(m)$, where $D(m)\in\mathcal{R}$. Let $V\subset [0,1]$ be the set on which the incremental condition $(vii)$ in Definition \ref{stmart} does not hold. Then $L(\lambda_{\nu})(st^{-1}(V))=0$, and we can choose $E_{m}\in\mathcal{C}_{\nu}$, with $\lambda_{\nu}(E_{m})={1\over m}$, such that $E_{m}\supset st^{-1}(V)$.  Then, we have that, for all $t\in{\mathcal{T}_{\nu}\setminus E_{m}}$, for all $K\leq 2D(m)m$, that;\\

${^{\circ}\int_{[\overline{X}]_{t,m}>K}1_{C_{m}}[\overline{X}]_{t,m}d\mu_{\eta}}$\\

$={^{\circ}\int_{\overline{\Omega}_{\eta}}1_{([\overline{X}]_{t,m}>K)\cap C_{m}}[\overline{X}]_{t,m}d\mu_{\eta}}$\\

$=\int_{\overline{\Omega}_{\eta}}1_{([\overline{X}]_{t,m}>K)\cap C_{m}}[X]_{^{\circ}t,m}dL\mu_{\eta}$\\

It follows that;\\

$\int_{[\overline{X}]_{t,m}>K}1_{C_{m}}[\overline{X}]_{t,m}d\mu_{\eta}$\\

$<\int_{[X]_{{^{\circ}t},m}>K-1}1_{C_{m}}[X]_{{^{\circ}t},m}dL(\mu_{\eta})+{1\over m}$\\

$<\int_{[X]_{{^{\circ}t},m}|>K-1}[X]_{{^{\circ}t},m}dL(\mu_{\eta})+{1\over m}$\\

$=f^{*}(K-1)+{1\over m}$\\

where we have used condition $(vii)$ in the definition from footnote \ref{stmart}. The condition $(*)$ holds trivially when $K>2D(m)m$, as then;\\

$\int_{[\overline{X}]_{t,m}>K}1_{C_{m}}[\overline{X}]_{t,m}d\mu_{\eta}=0$\\

It follows that;\\

${^{*}\mathcal{R}}\models (\forall t\in{\mathcal{T}_{\nu}\setminus E_{m}})(\forall K,)$\\

$\int_{[\overline{X}]_{t,m}>K}1_{C_{m}}[\overline{X}]_{t,m}d\mu_{\eta}<f^{*}(K-1)+{1\over m}$\\

for all sufficiently large $m\in\mathcal{N}$. By overflow, we can satisfy the condition for the same infinite $\kappa\in{^{*}\mathcal{N}}$ as above. In particular, we obtain, for infinite $K$, $t\in{\mathcal{T}_{\nu}\setminus E_{\kappa}}$ that;\\

$\int_{[\overline{X}]_{t,\kappa}>K}1_{C_{\kappa}}[\overline{X}]_{t,\kappa}d\mu_{\eta}<f^{*}(K-1)+{1\over \kappa}\simeq 0$\\

It follows, using the criterion in Lemma 3.19 of \cite{dep}, that $1_{C_{\kappa}}[\overline{X}]_{t,\kappa}\in SL^{1}(\overline{\Omega}_{\eta},\mu_{\eta})$, for all $t\in{\mathcal{T}_{\nu}\setminus E_{\kappa}}$ as required. Letting $D=C_{\kappa}$ $E=E_{\kappa}$ and noting that $\mu_{\eta}(D)=1-{1\over\kappa}\simeq 1$, $\mu_{\eta}(E)={1\over\kappa}\simeq 0$ we obtain the result.\\

\end{proof}

\begin{defn}
\label{integrand}
Let $\overline{X}$ be as in Definition \ref{mart}, with $E_{\eta}(\overline{X}_{0})=0$, and let $\{c_{j}(t,x):1\leq j\leq\nu\}$ be given as in Lemma \ref{rep}. Then we define;\\

 $\overline{H}:\overline{\mathcal{T}_{\nu}}\times\overline{\Omega_{\eta}}\rightarrow{^{*}{\mathcal C}}$, $\overline{Z}:\overline{\Omega_{\eta}}\rightarrow{^{*}{\mathcal C}}$, $\overline{Y}:\overline{\mathcal{T}_{\kappa}}\times\overline{\Omega_{\eta}}\rightarrow{^{*}{\mathcal C}}$,
 $\overline{W}:\overline{\mathcal{T}_{\kappa}}\times\overline{\Omega_{\eta}}\rightarrow{^{*}{\mathcal C}}$,
 $\{d_{j}(t,x):1\leq j\leq\nu\}$, $\overline{S}:\overline{\mathcal{T}_{\nu}}\times\overline{\Omega_{\eta}}\rightarrow{^{*}{\mathcal C}}$, $\overline{Q}:\overline{\Omega_{\eta}}\rightarrow{^{*}{\mathcal C}}$  by;\\

$\overline{H}(t,x)=\sqrt{\nu}c_{[\nu t]+1}(s,x)$\\

where $s\geq {[\nu t]+1\over \nu}$, for $0\leq t<1$ and;\\

$\overline{H}(t,x)=0$, for $t=1$\\

$\overline{Z}(x)={^{*}\sum}_{0\leq j\leq\nu-1}(\overline{X}_{j+1\over\nu}(x)-\overline{X}_{j\over\nu}(x))^{2}$\\

$\overline{Y}(t,x)=0$, for $0\leq [\nu t]<{\nu\over\kappa}-1$\\

$\overline{Y}(t,x)={k\over\nu}(\overline{H}^{2}_{[\nu t]\over\nu}+\overline{H}^{2}_{{[\nu t]-1\over\nu}}+\ldots+\overline{H}^{2}_{{[\nu t]-{\nu\over\kappa}+1\over\nu}})$, for ${\nu\over\kappa}-1\leq [\nu t]\leq 1$\\

$\overline{W}=\sqrt{\overline{Y}}$\\

$d_{j}(s,x)={1\over\sqrt{\nu}}\overline{W}_{j-1\over\nu}(x)$, for $1\leq j\leq \nu$, and ${j\over\nu}\leq s\leq 1$.\\

$\overline{S}(t,x)={{^{*}}\sum}_{j=1}^{[\nu t]}d_{j}(1,x)\omega_{j}$\\

$\overline{Q}(x)={^{*}\sum}_{0\leq j\leq\nu-1}(\overline{S}_{j+1\over\nu}(x)-\overline{S}_{j\over\nu}(x))^{2}$\\

\end{defn}

\begin{lemma}
\label{Sinteg}
If $\overline{X}$ is as in Lemma \ref{lift}, and $X$ is tame, then $\overline{Y}\in SL^{1}(\mathcal{T}_{\nu}\times\overline{\Omega}_{\eta},\lambda_{\nu}\times \mu_{\eta})$, $\overline{Z}\in SL^{1}(\overline{\Omega}_{\eta},\mu_{\eta})$ and $\overline{S}$ is a nonstandard martingale, with $\overline{S}_{1}\in SL^{2}(\overline{\Omega}_{\eta},\mu_{\eta})$.\\

\end{lemma}

\begin{proof}

The fact that $\overline{Z}\in SL^{1}(\overline{\Omega}_{\eta},\mu_{\eta})$, $(\dag)$, follows from Proposition 4.4.3 of \cite{AFKL} and the properties of $\overline{X}$. This does not require that $\overline{X}$ is $S$-continuous or has infinitesimal increments. \\

For the last claim, it is easily seen that the functions $d_{j}:[{j\over\nu},1]\times\overline{{\Omega}_{\eta}}\rightarrow{^{*}\mathcal{C}}$ are $\mathcal{D}_{\nu}\times\mathcal{C}_{\eta}^{j-1}$-measurable, for $1\leq j\leq \nu$. Hence, using Lemma \ref{rep}, we have that $\overline{S}$ satisfies conditions $(i)$ and $(ii)$ of Definition \ref{mart}. By Proposition 4.4.3 of \cite{AFKL}, it is sufficient to show that $\overline{Q}\in SL^{1}(\overline{\Omega}_{\eta},\mu_{\eta})$, as $\overline{S}_{0}=0$. (explain why we can assume this?) We compute;\\

$\overline{Q}(x)={^{*}\sum}_{0\leq j\leq\nu-1}(\overline{S}_{j+1\over\nu}(x)-\overline{S}_{j\over\nu}(x))^{2}$\\

$={^{*}\sum}_{1\leq j\leq\nu-1}d_{j}^{2}(1,x)$\\

$={1\over\nu}{^{*}\sum}_{1\leq j\leq\nu-1}\overline{W}_{j-1\over\nu}^{2}(x)$\\

$={1\over\nu}{^{*}\sum}_{1\leq j\leq\nu-1}\overline{Y}_{j-1\over\nu}(x)$\\

$={1\over\nu}{^{*}\sum}_{{\nu\over\kappa}-1\leq j\leq\nu-2}{k\over\nu}(\overline{H}^{2}_{j\over\nu}+\overline{H}^{2}_{{j-1\over\nu}}+\ldots+\overline{H}^{2}_{{j-{\nu\over\kappa}+1\over\nu}})$\\

$={1\over \nu}{^{*}\sum}_{0\leq j\leq \nu-1}\overline{H}^{2}_{{j\over\nu}}d\mu_{\eta}+r(x)$\\

$={^{*}\sum}_{1\leq j\leq \nu}c^{2}_{j}(1,x)d\mu_{\eta}+r(x)$\\

$={^{*}\sum}_{0\leq j\leq\nu-1}(\overline{X}_{j+1\over\nu}(x)-\overline{X}_{j\over\nu}(x))^{2}+r(x)$\\

$=\overline{Z}(x)+r(x)$\\

where $r(x)\geq 0$ is a remainder term. We have that $E_{\eta}(r(x))\simeq 0$, and $r(x)\leq \overline{Z}(x)$. It follows, easily, that $r(x)\simeq 0$, a.e $L(\mu_{\eta})$, $\overline{Q}(x)\simeq \overline{Z}(x)$ a.e $L(\mu_{\eta})$, and $\overline{Q}(x)\in SL^{2}(\overline{\Omega}_{\eta},\mu_{\eta})$ as required.\\

For the first part, observe first that $\overline{H}$ is progressively measurable, that is $\overline{H}_{t}$ is measurable with respect to $\mathcal{C}_{\eta}^{[\nu t]}$, hence, so is $\overline{Y}$.\\

By Lemma 3.19 of \cite{dep}, it is sufficient to prove that;\\

$\int_{\overline{Y}>K}\overline{Y}d(\lambda_{\nu}\times \mu_{\eta})\simeq 0$, for $K$ infinite\\

As $\overline{Y}$ is progressively measurable, the set $\overline{Y}>K$ is progressively measurable. Moreover, it has infinitesimal measure. This clearly follows from showing that;\\

$\int_{\mathcal{T}_{\nu}\times\overline{\Omega}_{\eta}}\overline{Y}d\lambda_{\nu}d\mu_{\eta}$ is finite, $(*)$\\

To see $(*)$, we compute;\\

$\int_{\mathcal{T}_{\nu}\times\overline{\Omega}_{\eta}}\overline{Y}(t,x)d\lambda_{\nu}d\mu_{\eta}$\\

$={1\over \nu}{^{*}\sum}_{0\leq j\leq \nu-1}\int_{\overline{\Omega}_{\eta}}\overline{Y}({j\over\nu},x)d\mu_{\eta}$\\

$={1\over \nu}{^{*}\sum}_{{\nu\over k}-1\leq j\leq \nu-1}\int_{\overline{\Omega}_{\eta}}\overline{Y}({j\over\nu},x)d\mu_{\eta}$\\

$={1\over \nu}{^{*}\sum}_{{\nu\over k}-1\leq j\leq \nu-1}\int_{\overline{\Omega}_{\eta}}({k\over\nu}(\overline{H}^{2}_{{j\over\nu}}+\overline{H}^{2}_{{j-1\over\nu}}+\ldots+\overline{H}^{2}_{{j-{\nu\over\kappa}+1\over\nu}}))d\mu_{\eta}$\\

$\leq {1\over \nu}{^{*}\sum}_{0\leq j\leq \nu}\int_{\overline{\Omega}_{\eta}}\overline{H}^{2}_{{j\over\nu}}d\mu_{\eta}$\\

$={1\over \nu}{^{*}\sum}_{0\leq j\leq \nu-1}\int_{\overline{\Omega}_{\eta}}\nu|c_{j}(1,x)|^{2}d\mu_{\eta}$ $(\dag\dag)$\\

$={^{*}\sum}_{0\leq j\leq \nu-1}\int_{\overline{\Omega}_{\eta}}|c_{j}(1,x)|^{2}d\mu_{\eta}$\\

$=\int_{\overline{\Omega}_{\eta}}|\overline{X}_{1}|^{2}d \mu_{\eta}$  $(\dag\dag\dag)$\\

where, in $(\dag\dag)$, we have used Definition \ref{integrand}, and, in $(\dag\dag\dag)$, we have used the fact that $\overline{X}_{1}={^{*}{\sum}}_{0\leq j\leq \nu-1}c_{j}(1,x)\omega_{j}$, by Lemma \ref{rep}, and the orthogonality observation $(*)$ there. Hence, $(*)$ is shown, by the assumption that $\overline{X}_{1}\in SL^{2}(\overline{\Omega}_{\eta},\mu_{\eta})$. Therefore, it is sufficient to prove that;\\

$\int_{A}\overline{Y}d(\lambda_{\nu}\times \mu_{\eta})\simeq 0$, for a progressively measurable set $A$ with $\lambda_{\nu}\times \mu_{\eta}(A)\simeq 0$. $(**)$ \\

We now verify $(**)$;\\

Case 1. Let $A\subset\overline{{\Omega}_{\eta}}$, with $\mu_{\eta}(A)\simeq 0$, then;\\

$\int_{A\times\overline{\mathcal T}_{\eta}}\overline{Y} d\mu_{\eta}d\lambda_{\nu}$\\

$={1\over\nu}{^{*}\sum}_{0\leq j\leq \nu-1}\int_{A}\overline{Y}d\mu_{\eta}$\\

$\leq {1\over\nu}{^{*}\sum}_{0\leq j\leq \nu-1}\int_{A}\overline{H}^{2}_{{j\over\nu}} d\mu_{\eta}$ (as above)\\

$=\int_{A}{^{*}\sum}_{0\leq j\leq \nu-1}c_{j}(1,x)^{2}d\mu_{\eta}$\\

$=\int_{A}{^{*}\sum}_{0\leq j\leq \nu-1}(\overline{X}_{j+1\over\nu}-\overline{X}_{j\over\nu})^{2}d\mu_{\eta}
=\int_{A}\overline{Z}\simeq 0$\\

by $(\dag)$.\\

Case 2. Let $B\subset\overline{T}_{\nu}$, with $B\in\mathcal{D}_{\nu}$ and $\lambda_{\nu}(B)\simeq 0$. We can write $B=\bigcup_{1\leq j\leq s}I_{j}$, where $I_{j}$ is an interval of the form $[{i_{j}\over\nu},{i_{j}+1\over\nu})$, for some $0\leq i_{j}\leq\nu-1$, and ${s\over \nu}\simeq 0$. We compute, for $i_{j}\geq {\nu\over\kappa}-1$;\\

$\int_{\overline{\Omega}_{\eta}\times I_{j}}\overline{Y}(t,x)d\lambda_{\nu}d\mu_{\eta}$\\

$={1\over\nu}\int_{\overline{\Omega}_{\eta}}({k\over\nu}(\overline{H}^{2}_{{i_{j}\over\nu}}+\overline{H}^{2}_{{i_{j}-1\over\nu}}+\ldots+\overline{H}^{2}_{{i_{j}-{\nu\over\kappa}+1\over\nu}}))d\mu_{\eta}$\\

$={k\over\nu}\int_{\overline{\Omega}_{\eta}}(c_{i_{j}+1}^{2}+\ldots c_{i_{j}-{\nu\over\kappa}+2}^{2})d\mu_{\eta}$\\

We have that;\\

$\overline{X}(t,x)=\sum_{0\leq j\leq {t\nu}}c_{j}(1,x)\omega_{j}(x)$\\

$\overline{X}(t,x)^{2}=\sum_{0\leq j,k\leq t\nu}c_{j}(1,x)c_{k}(1,x)\omega_{j}(x)\omega_{k}(x)$ $(\sharp)$\\

Then;\\

$\int_{\overline{\Omega}_{\eta}}(\overline{X}_{t})^{2}(x)d\mu_{\eta}$\\

$=\sum_{0\leq j,k\leq [t\nu] }\int_{\overline{\Omega}_{\eta}}c_{j}(1,x)c_{k}(1,x)\omega_{j}\omega_{k}d\mu_{\eta}$ (using $(\sharp)$)\\

$=\sum_{0\leq j\leq [t\nu]}\int_{\overline{\Omega}_{\eta}}c_{j}^{2}(1,x)d\mu_{\eta}$ (using Lemma \ref{rep}) $(\sharp\sharp)$\\

It follows that;\\

$\int_{\overline{\Omega}_{\eta}}(c_{i_{j}+1}^{2}+\ldots c_{i_{j}-{\nu\over\kappa}+2}^{2})d\mu_{\eta}$\\

$=\int_{\overline{\Omega}_{\eta}}(\overline{X}^{2}_{{i_{j}+1\over\nu}}-\overline{X}^{2}_{{i_{j}+1-{\nu\over\kappa}\over\nu}})d\mu_{\eta}$\\

and, therefore, that;\\

$\int_{\overline{\Omega}_{\eta}\times I_{j}}\overline{Y}(t,x)d\lambda_{\nu}d\mu_{\eta}$\\

${\kappa\over\nu}\int_{\overline{\Omega}_{\eta}}(\overline{X}^{2}_{{i_{j}+1\over\nu}}-\overline{X}^{2}_{{i_{j}+1\over\nu}-{1\over\kappa}})d\mu_{\eta}\leq {C+1\over \nu}$\\

using Lemma \ref{tame}. We then have that;\\

$\int_{\overline{\Omega}_{\eta}\times B}\overline{Y}(t,x)d\lambda_{\nu}d\mu_{\eta}$\\

$={^{*}\sum}_{1\leq j\leq s}\int_{\overline{\Omega}_{\eta}\times I_{j}}\overline{Y}(t,x)d\lambda_{\nu}d\mu_{\eta}\leq {s(C+1)\over\nu}\simeq 0$\\

as required.\\

Case 3. Let $B\in\mathcal{D}_{\nu}\times\mathcal{C}_{\eta}$, with $(\lambda_{\nu}\times \mu_{\eta})(B)=\delta\simeq 0$. Let;\\

 $I=\{i:0\leq i\leq \nu,\mu_{\eta}(pr_{\eta}(B\cap pr_{\nu}^{-1}({i\over \nu})))>\delta^{1\over 2}\}$\\

Let $C=\bigcup_{i\in I}[{i\over\nu},{i+1\over\nu})$, so $C\in\mathcal{D}_{\nu}$, and let $B_{1}=B\cap pr_{\nu}^{-1}(C)$. As $B_{1}\subset B$, and by construction of $C$, we have that;\\

$\delta\geq (\lambda_{\nu}\times \mu_{\eta})(B_{1})>\delta^{1\over 2}\lambda_{\nu}(C)$\\

It follows that $\lambda_{\nu}(C)<\delta^{1\over 2}\simeq 0$. By Case 2, we have that;\\

$\int_{B_{1}}\overline{Y}d(\lambda_{\nu}\times \mu_{\eta})$\\

$\leq \int_{\overline{\Omega}_{\eta}\times C}\overline{Y}d(\lambda_{\nu}\times \mu_{\eta})\simeq 0$\\

Let $B_{2}=B\cap B_{1}^{c}$, then $B_{2}\in\mathcal{D}_{\nu}\times\mathcal{C}_{\eta}$ and $(\lambda_{\nu}\times \mu_{\eta})(B_{2})\simeq 0$, and to show Case 3, it is sufficient to prove that;\\

$\int_{B_{2}}\overline{Y}d(\lambda_{\nu}\times \mu_{\eta})\simeq 0$\\

We say that $B\in \mathcal{D}_{\nu}\times\mathcal{C}_{\eta}$ is wide, if there exists $\epsilon\simeq 0$, with $\mu_{\eta}(pr_{\eta}(B\cap pr_{\nu}^{-1}(t)))\leq\epsilon$, for $t\in\mathcal{T}_{\nu}$, and note that $B_{2}$ is wide. We are thus reduced to;\\

Case 4. Suppose $B$ is progressively measurable and wide, and let;\\

 $I_{j}=\{i\in{^{*}\mathcal{N}}:0\leq i\leq \nu-1, rem(2,i)=j, B\cap pr^{-1}_{\nu}({i\over\nu})\neq\emptyset\}$, for $0\leq j\leq 1$\\

$S_{j}=\bigcup_{i\in I_{j}}[{i\over\nu},{i+1\over\nu})$, $0\leq j\leq 1$\\

$B_{j}=B\cap pr^{-1}_{\nu}(S_{j})$, $0\leq j\leq 1$\\

Then $B=\bigcup_{0\leq j\leq 1}B_{j}$, and each $B_{j}$ is progressively measurable and wide. Let;\\

 $V_{j}=\{(i,s)\in{^{*}\mathcal{N}}^{2}:1\leq i\leq \nu-1, 0\leq s<2^{i}, rem(2,s)=j, B\cap pr^{-1}_{\nu}({i\over\nu})\neq\emptyset, B\cap pr^{-1}_{\eta}({s\over\eta})\neq \emptyset\}$, for $0\leq j\leq 1$\\

 $W_{j}=\bigcup_{(i,s)\in V_{j}}[{i\over\nu},{i+1\over\nu})\times [{s\over 2^{i}},{s+1\over 2^{i}})$, $0\leq j\leq 1$\\

 By the progressive measurability of $B$, $B=\bigcup_{0\leq j\leq 1}W_{j}$ and each $W_{j}$ is progressively measurable and wide. Let $B_{ij}=B_{i}\cap W_{j}$, $0\leq i\leq 1$, $0\leq j\leq 1$. Then $B=\bigcup_{0\leq i,j\leq 1}B_{ij}$ and each $B_{ij}$ is progressively measurable and wide. We say that $B\in \mathcal{D}_{\nu}\times\mathcal{C}_{\eta}$ is separated if, for all $(t,x)\in B$, $(t+{1\over\nu})\notin pr_{\nu}(B)$, and $(t,{[x2^{[t\nu]}]+1\over 2^{[t\nu]}})\notin B$, for $[\nu t]\geq 1$ and $0\leq [[x2^{[t\nu]}]\leq 2^{[t\nu]}-2$. . By construction, each $B_{ij}$ is separated, for $0\leq i,j\leq 1$. We are thus reduced to;\\

Case 5. Suppose $B$ is progressively measurable, wide and separated.\\

Observe that;\\

 $\kappa([\overline{X}]_{t}-[\overline{X}]_{t-{1\over\kappa}})$\\

 $=\kappa(\sum_{j=0}^{[\nu t]-1}(\overline{X}_{j+1\over\nu}-\overline{X}_{j\over\nu})^{2}-\sum_{j=0}^{[\nu (t-{1\over\kappa})]-1}(\overline{X}_{j+1\over\nu}-\overline{X}_{j\over\nu})^{2})$\\

 $=\kappa(\sum_{j=[\nu t]-{\nu\over\kappa}}^{[\nu t]-1}(c_{j+1})^{2})$\\

 $={\kappa\over\nu}(\sum_{j=[\nu t]-{\nu\over\kappa}}^{[\nu t]-1}(\overline{H}_{j\over\nu})^{2})$\\

 $=\overline{Y}_{t-{1\over\nu}}$\\

 It follows from Lemma \ref{tame}, that there exists $E'$ with $\mu_{\eta}(E')=0$, such that $1_{D}\overline{Y}_{t}\in SL^{1}(\overline{\Omega}_{\eta},\mu_{\eta})$, $(\dag\dag)$ for all $t\in{\mathcal{T}_{\nu}\setminus E'}$. We now compute;\\

 $\int_{B}\overline{Y}d(\lambda_{\nu}\times \mu_{\eta})$\\

 $\leq\int_{B\cap (D^{c}\times \mathcal{T}_{\nu})}\overline{Y}d(\lambda_{\nu}\times\mu_{\eta})+\int_{B\cap (\overline{\Omega}_{\eta}\times E')}\overline{Y}d(\lambda_{\nu}\times\mu_{\eta})$\\
 
 $+\int_{B\cap (D\times {\mathcal{T}_{\nu}\setminus E'})}\overline{Y}d(\lambda_{\nu}\times\mu_{\eta})$\\

 $\simeq \int_{B\cap (D\times {\mathcal{T}_{\nu}\setminus E'})}\overline{Y}d(\lambda_{\nu}\times\mu_{\eta})$ (by Cases 1,2)\\

 $=\int_{{\mathcal{T}_{\nu}\setminus E'}}\int_{\overline{\Omega}_{\eta}}1_{D}\overline{Y}_{t}d\mu_{\eta}d\lambda_{\nu}$\\

 $=\int_{{\mathcal{T}_{\nu}\setminus E'}}g(t)d\lambda_{\nu}$ (where $g\simeq 0$ on ${\mathcal{T}_{\nu}\setminus E'}$)\\

 $\simeq 0$\\

 where we have used the assumption $(\dag\dag)$ and the fact that $B$ is wide in the penultimate line. It follows that $\overline{Y}\in SL^{1}(\overline{\Omega}_{\eta}\times\mathcal{T}_{\nu})$ as required.\\

\end{proof}

\begin{theorem}
\label{final}
Any tame martingale $X$ is representable as a stochastic integral;\\

$X(t,x)=\int_{0}^{t}F(s,x) d\beta_{s}$\\

where $F:[0,1]\times{\overline{\Omega}_{\eta}}\rightarrow{\mathcal{R}}\in L^{2}([0,1]\times\overline{\Omega}_{\eta},L(\mu_{\eta}))$, and $\beta_{s}$ is a Brownian motion.

\end{theorem}

\begin{proof}
By Lemma \ref{lift}, there exists a nonstandard martingale $\overline{X}$, with ${^{\circ}(\overline{X}_{t})}=X_{^{\circ}t}$, for $t\in\overline{\mathcal{T}}_{\nu}$, a.e $L(\mu_{\eta})$. Let notation be as in Definition \ref{integrand}. Then by Lemma \ref{Sinteg}, we have shown that $\overline{Y}\in SL^{1}(\overline{\mathcal{T}_{\nu}}\times\overline{\Omega}_{\eta})$. We have that $\overline{S}=\int \overline{W}d\chi$, , where $\chi$ is Anderson's random walk, and, therefore, the quadratic variation;\\

 $[\overline{S}]=\overline{Q}=\int \overline{W}^{2}dt$.\\

 We claim that;\\

${^{\circ}}[\overline{S}](x,t)=\int_{0}^{t}f(x,s)ds$ a.e $dL(\mu_{\eta})$ $(*)$\\

where $f\in L^{1}(\overline{\Omega}_{\eta}\times [0,1])$. To see this, we first claim that $\overline{W}_{x}^{2}\in SL^{1}(\mathcal{T}_{\nu})$ $a.e$ $dL(\mu_{\eta})$, $(**)$. Suppose not, then, using Theorem 9 of \cite{A}, there exists $A$ with $L(\mu_{\eta})(A)>0$, such that;\\

${^{\circ}}\int_{0}^{1}\overline{W}_{x}^{2}d\lambda_{\nu}>\int_{0}^{1}{^{\circ}}\overline{W}_{x}^{2}d L(\lambda_{\nu})$.\\

But then;\\

${^{\circ}}\int_{A}\int_{0}^{1}\overline{W}^{2}d\lambda_{\nu}d\mu_{\eta}$\\

$\geq \int_{A}{^{\circ}}\int_{0}^{1}\overline{W}^{2}d\lambda_{\nu}dL(\mu_{\eta})$\\

$>\int_{A}\int_{0}^{1}{^{\circ}}\overline{W}^{2}dL(\lambda_{\nu})dL(\mu_{\eta})$\\

contradicting the fact that $\overline{W}\in SL^{2}(\overline{\mathcal{T}_{\nu}}\times\overline{\Omega_{\eta}},\lambda_{\nu}\times\mu_{\eta})$. Hence, $(**)$ is shown. Let $V_{x}(t)=\int_{0}^{t}\overline{W}_{x}^{2}d\lambda_{\nu}$, for $t\in [0,1]$. By $(**)$, we have that;\\

${^{\circ}}V_{x}(t)=\int_{0}^{t}{^{\circ}}\overline{W}_{x}^{2}dL(\lambda_{\nu})$\\

 We claim that ${^{\circ}}V_{x}$ is absolutely continuous, $(***)$. Suppose not, then there exist internal $B_{n}\subset\mathcal{T}_{\nu}$, with each $B_{n}$ a finite union of intervals with real endpoints, such that $\lambda(B_{n}\cap [0,1])<{1\over n}$, where $\lambda$ is Lebesgue measure, and $\epsilon\in\mathcal{R}_{>0}$, such that;\\

 $\int_{B_{n}}{^{\circ}}\overline{W}_{x}^{2}dL(\lambda_{\nu})>\epsilon$\\

Then ${^{\circ}}\int_{B_{n}}\overline{W}_{x}^{2}d\lambda_{\nu}\geq \int_{B_{n}}{^{\circ}}\overline{W}_{x}^{2}dL(\lambda_{\nu}>\epsilon$\\

and $\lambda_{\nu}(B_{n})\simeq \lambda(B_{n}\cap [0,1])<{1\over n}$\\

as each $B_{n}$ is a finite union of intervals. We can extend the sequence $(B_{n})_{n\in\mathcal{N}}$ to an internal sequence indexed by ${^{*}\mathcal{N}}$. By overflow, we can find an infinite $\rho\in{{^{*}\mathcal{N}}\setminus\mathcal{N}}$, with $B_{\rho}\in\mathcal{D}_{\nu}$, such that $\lambda_{\nu}(B_{\rho})<{1\over\rho}\simeq 0$ and;\\

$\int_{B_{\rho}}\overline{W}_{x}^{2}d\lambda_{\nu}>\epsilon$\\

This contradicts $(**)$. Hence, $(***)$ is shown. By real analysis, see \cite{Rud} Theorem 7.18, the derivative $f_{x}=({^{\circ}}V_{x})'$ exists a.e $d\lambda$, $f_{x}\in L^{1}([0,1])$ and;\\

${^{\circ}}[\overline{S}](x,t)=\int_{0}^{t}f(x,s)ds$ a.e $dL(\mu_{\eta})$\\

We compute;\\

$\int_{\overline{\Omega}_{\eta}}\int_{0}^{1}f(x,s)ds$\\

$=\int_{\overline{\Omega}_{\eta}}\int_{\overline{\mathcal{T}}_{\nu}}{^{\circ}}\overline{W}^{2}dL(\lambda_{\nu})dL(\mu_{\eta})$\\

$={^{\circ}}\int_{\overline{\Omega}_{\eta}}\int_{\overline{\mathcal{T}}_{\nu}}\overline{W}^{2}d\lambda_{\nu}d\mu_{\eta}$\\

which is finite, as $\overline{W}\in SL^{2}(\overline{\mathcal{T}_{\nu}}\times\overline{\Omega_{\eta}},\lambda_{\nu}\times\mu_{\eta})$, hence $f\in  L^{1}({\overline{\Omega_{\eta}}}\times [0,1])$, thus $(*)$ is shown.\\

We have that;\\

$[\overline{S}]_{t}\simeq [\overline{X}]_{t}=\overline{Z}_{t}$ $a.e dL(\mu_{\eta}$\\

This follows by computing the remainder term $r(x)$ in the proof of Lemma \ref{Sinteg} and using the fact that $\overline{Z}$ is $S$-continuous. This last is a consequence of the fact that $\overline{X}$ is $S$-continuous and $\overline{X}_{1}\in SL^{2}(\overline{\Omega}_{\eta})$, using Theorem 4.2.16 of \cite{AFKL}. Hence, we have;\\

${^{\circ}}[\overline{X}](x,t)=\int_{0}^{t}f(x,s)ds$ a.e $dL(\mu_{\eta})$ $(****)$\\

Define a new adapted process $g$ by;\\

$g(x,t)=f^{{-1\over 2}}(x,t)$ if $f(x,t)\neq 0$, and $g(x,t)=0$ otherwise.\\

Let $1_{g}$ be the characteristic function of the set $\{(x,t):g(x,t)=0\}$. We have that;\\

$E(\int_{0}^{1}g(x,s)^{2}d{^{\circ}} [\overline{X}])=E(\int_{0}^{1}g(x,s)^{2}f(x,s)ds)\leq 1$\\

hence, $g\in L^{2}(\nu_{{^{\circ}} [\overline{X}]})$. Let $G\in SL^{2}(\overline{X})$ be a 2-lifting of $g$, and $1_{G}$ a 2-lifting of $1_{g}$. We can assume that $G.1_{G}=0$. Define;\\

$\beta(x,t)={^{\circ}}(\int_{0}^{t}G(x,s)d\overline{X}(x,s)+\int_{0}^{t}1_{G}(x,s)d\chi(x,s))$\\

Since, $G$ and $1_{G}$ have disjoint supports;\\

$[\beta](x,t)={^{\circ}}[\int Gd\overline{X}](x,t)+{^{\circ}}[\int 1_{G}d\chi](x,t)$\\

$={^{\circ}}(\int G^{2}d\overline{X})(x,t)+{^{\circ}}[\int 1_{G}dt](x,t)$\\

$=\int_{0}^{t}g^{2}fds+\int_{0}^{t}1_{g}^{2}ds=\int_{0}^{t}1ds=t$\\

It follows, using Proposition 4.4.13 and 4.4.18 of \cite{AFKL}, this requires that $\overline{X}$ has infinitesimal increments, that $\beta$ is a Brownian motion, adapted to the filtration $(\overline{\Omega}_{\eta},\mathcal{D}_{t},L(\mu_{\eta})$. We have that $f^{{1\over 2}}\in L^{2}(\nu_{\beta})$ and;\\

$\int f^{{1\over 2}}d\beta=\int f^{{1\over 2}} gd{^{\circ}}\overline{X}+\int f^{{1\over 2}}1_{g}d{^{\circ}}\chi=\int f^{{1\over 2}}gd{^{\circ}}\overline{X}$\\

since  $f^{{1\over 2}}1_{g}=0$. It remains to show that ${^{\circ}}\overline{X}=\int f^{{1\over 2}}gd{^{\circ}}\overline{X}$, since, we then get the result by setting $F=f^{{1\over 2}}$. Using Doob's inequality;\\

$E(sup_{q\leq 1,q\in\mathcal{Q}}({^{\circ}}\overline{X}(q)-\int_{0}^{q}f^{{1\over 2}}gd{^{\circ}}\overline{X})^{2})$\\

$\leq 4E(({^{\circ}}\overline{X}(1)-\int_{0}^{1}f^{{1\over 2}}gd{^{\circ}}\overline{X})^{2})$\\

$=4E(\int_{0}^{1}(1-f^{{1\over 2}}g)^{2}d{^{\circ}}\overline{X})$\\

$=4E(\int_{0}^{1}(1-f^{{1\over 2}}g)^{2}dt)=0$\\

as $f^{{1\over 2}}g=1$, whenever $f\neq 0$.\\

\end{proof}

\end{document}